\documentclass[11pt]{amsart}

\usepackage{anysize} \marginsize{1.3in}{1.3in}{1in}{1in}
\usepackage{comment}
\usepackage{xcolor}
\usepackage{amsmath}
\usepackage{mathtools}
\usepackage[all]{xy}
\usepackage[utf8]{inputenc}
\usepackage{varioref}
\usepackage{amsfonts}
\usepackage{amssymb}
\usepackage{bbm}
\usepackage{esint}
\usepackage{graphicx}
\usepackage{tikz}
\usepackage{empheq}
\usepackage{enumitem}
\usepackage{tikz-cd}
\usetikzlibrary{matrix,arrows,decorations.pathmorphing}
\usepackage{mathrsfs}
\usepackage[hypertexnames=false,backref=page,pdftex,
 	pdfpagemode=UseNone,
 	breaklinks=true,
 	extension=pdf,
 	colorlinks=true,
 	linkcolor=blue,
 	citecolor=red,
 	urlcolor=blue,
 ]{hyperref}

\theoremstyle{plain}

\newtheorem{theorem}[subsection]{Theorem}

\newtheorem{lemma}[subsection]{Lemma}
\newtheorem{corollary}[subsection]{Corollary}

\newtheorem{proposition}[subsection]{Proposition}

\theoremstyle{remark}
\newtheorem{example}[subsection]{Example}

\theoremstyle{definition}

\numberwithin{equation}{section}

\theoremstyle{remark}
\newtheorem{remark}[subsection]{Remark}

{\theoremstyle{definition}
}

\title[On the period of Li, Pertusi and Zhao's symplectic variety]{On the period of Li, Pertusi and Zhao's symplectic variety}

\author{Franco Giovenzana}
\address[Franco Giovenzana]{Fakult\"at f\"ur Mathematik\\ Technische Universit\"at Chemnitz\\
Reichenhainer Stra\ss e 39, 09126 Chemnitz, Germany}
\email{franco.giovenzana@math.tu-chemnitz.de}

\author{Luca Giovenzana}

\address[Luca Giovenzana]{Mathematical Sciences, Loughborough University, Schofield Building, Epinal Way,
Loughborough,
Leicestershire,
LE11 3TU}
\email{L.Giovenzana@lboro.ac.uk}
\curraddr{Fakult\"at f\"ur Mathematik\\ Technische Universit\"at Chemnitz\\
Reichenhainer Stra\ss e 39, 09126 Chemnitz, Germany}
\email{luca.giovenzana@math.tu-chemnitz.de}

\author{Claudio Onorati}
\address[Claudio Onorati]{Dipartimento di Matematica, Universit\`a di Roma Tor Vergata, via della ricerca scientifica 1, 00133, Roma, Italia}
\email{onorati@mat.uniroma2.it}

\let\origmaketitle\maketitle
\def\maketitle{
  \begingroup
  \def\uppercasenonmath##1{} 
  \let\MakeUppercase\relax 
  \origmaketitle
  \endgroup
}

\newcommand{\Ktop}{{\operatorname{K_{\mathrm{top}}}}}

\newcommand{\sJ}{{\mathcal J}}

\newcommand{\sO}{{\mathcal O}}

\newcommand{\sU}{{\mathcal U}}

\newcommand{\C}{{\mathbb C}}

\renewcommand{\P}{{\mathbb P}}
\newcommand{\Q}{{\mathbb Q}}

\newcommand{\Z}{{\mathbb Z}}

\newcommand{\Db}{\operatorname{D}}

\DeclareMathOperator{\oH}{H}

\newcommand{\Pic}{\operatorname{Pic}}

\newcommand{\Stab}{{\operatorname{Stab}}}

\newcommand{\Sym}{{\rm Sym}}

\renewcommand{\to}[1][]{\xrightarrow{\ #1\ }}

\begin{document}

\begin{abstract}
We extend classical results of Perego and Rapagnetta on moduli spaces of sheaves of type OG10 to moduli spaces of Bridgeland semistable objects on the Kuznetsov component of a cubic fourfold. In particular, we determine the 
period of this class of varieties and use it to understand when they become birational to moduli spaces of sheaves on a K3 surface.
\end{abstract}

\subjclass[2010]{}
\keywords{Moduli spaces of Bridgeland semistable objects, cubic fourfolds, intermediate jacobian fibrations, irreducible holomorphic symplectic manifolds}

\maketitle

\tableofcontents

\section*{Introduction}\label{section introduction}

The theory of moduli spaces of sheaves on algebraic surfaces is arguably one of the most important and fruitful research areas of the last decades. In particular, moduli spaces of sheaves on K3 surfaces have been extensively studied in relation to the geometry of irreducible holomorphic symplectic manifolds, i.e.\ simply connected compact K\"ahler manifolds with a unique up to scalar non-degenerate holomorphic $2$-form.

Let $(S,H)$ be a polarised K3 surface, $v=v(F)$ be the Mukai vector of a coherent sheaf $F$ on $S$ and $M$ be the moduli space of Gieseker semistable sheaves on $S$ of Mukai vector $v$.
It is known that when $v$ is primitive, i.e.\ not of the form $v=kw$ with $k\neq\pm1$, and $v^2\geq-2$ and $H$ is generic with respect to $v$, then the moduli space is nonempty, irreducible, and symplectic (see \cite{Yoshioka:ModuliSpaces} for the final statement of a long list of works). In this case, any semistable sheaf is stable. On the other hand, if the Mukai vector is not primitive, then the singular locus of $M$ coincides with the strictly semistable locus, which may be nonempty.
By the seminal work of Mukai (\cite{Mukai}), the stable locus always carries a holomorphic symplectic form, but in general there are no symplectic resolutions of singularities. More precisely, if the polarisation is $v$-general, we have that (\cite{LehnSorger,KLS06}):
\begin{itemize}
\item if $v=2w$ with $w^2=2$, then there exists a symplectic resolution of singularities;
\item if $v=kw$ with $w^2\neq2$ or $w^2=2$ and $k\geq3$, then there are no symplectic resolutions of singularities.
\end{itemize}
In the former case, which we are interested in, the desingularisation $\widetilde{M}$ is an irreducible holomorphic symplectic manifold of type OG10. Its geometry is quite well understood nowadays, in particular its Hodge structure is abstractly described by Perego and Rapagnetta (\cite{PR:Factoriality}). We recall these works in Section~\ref{section:sheaves}.

The first aim of this paper is to generalise these results to the non-commutative case. Since Bridgeland's work on stability conditions (\cite{Bridgeland:Primo}), the realm of moduli spaces welcomed more general objects, namely moduli spaces parametrising semistable objects in the derived category of coherent sheaves on a K3 surface.
Even more generally, stability conditions have been recently constructed on some K3 categories (often called non-commutative K3 surfaces).
Examples of these are given by the Kuznetsov component of a smooth cubic fourfold.
This level of generality is the one we consider in this paper. The machinery to study these objects has been rigorously developed in \cite{BLNPS}, where the case of a primitive Mukai vector is comprehensively analysed.

Recently Li, Pertusi and Zhao studied moduli spaces of Bridgeland semistable objects in the K3 category of a cubic fourfold with a non primitive Mukai vector of the form $v=2w$, with $w^2=2$ (see \cite{PertusiCo} -- we recall some of their work in Section~\ref{section:BBF LPZ}). These varieties are called \emph{LPZ varieties} in the following. As in the classical case, they show that a symplectic resolution of singularities exists and the resulting manifold is an irreducible holomorphic symplectic manifold of type OG10. Paralleling the classical works of Perego and Rapagnetta (\cite{PR13,PR:Factoriality}), we determine the periods of both the singular moduli space and its desingularisation, see Proposition~\ref{prop:v perp} and Proposition~\ref{prop:Gamma}.

As a first corollary we get a Torelli-like statement that compares the birational geometry of LPZ varieties with a particular Mukai vector (see Example~\ref{example:LPZ}) with the geometry of the cubic fourfold. More precisely, in Theorem~\ref{thm:Torelli} we prove that the existence of a birational morphism between two such varieties, satisfying two additional assumptions, implies that the underlying cubic fourfolds are isomorphic.

As a second corollary we get the following analog of the main result of \cite{PR:Factoriality}:
\begin{center}
\emph{An LPZ variety is either locally factorial or 2-factorial.}
\end{center}
See Corollary~\ref{cor:factoriality} for the precise statement.

In the rest of the paper we use the results of Section~\ref{section:BBF LPZ} to investigate the following natural question.

\begin{center}
\underline{{\bf Q}}: 
When is an LPZ variety birational to a moduli space of sheaves on a (twisted) K3 surface?
\end{center}

We give a complete answer to this question in the non-twisted case in Theorem~\ref{thm:qualsiasi}. As expected, such a birational isomorphism exists as soon as the cubic fourfold has an associated K3 surface. When the K3 surface is twisted though, we need to rigidify our hypothesis, and we prove a similar statement under the assumption that the birational map is stratum preserving in Proposition~\ref{thm:main}. Here, a stratum preserving birational map is a birational map that is well-defined at the generic point of the singular locus of the first variety and maps it to the generic point of the singular locus of the second variety.

The same question for Fano varieties of lines of cubic fourfolds and for the so-called Lehn--Lehn--Sorger--van Straten symplectic eightfolds has been previously answered, by similar methods, in \cite{Addington:Fano, huybrechts-twisted} and \cite{AG:LLSvS, LPZ:LLSvS}. From this point of view, the present work can be thought as a natural continuation of the aforementioned works. Recently, similar results of the ones in the present paper appeared in \cite{FGG}.

Finally, in \cite[Theorem~1.3]{PertusiCo} it is proved that certain LPZ varieties are birational to an irreducible holomorphic symplectic manifold of type OG10 that compactifies the twisted intermediate jacobian fibration associated to a cubic fourfold (see \cite{Voisin:Twisted}). Using this remark, in Theorem~\ref{thm:LPZ is LSV} we give an answer to the following question.

\begin{center}
\underline{{\bf Q}}: 
When is an LPZ variety birational to an LSV variety?
\end{center}

An LSV variety is an irreducible holomorphic symplectic manifold of type OG10 that compactifies the intermediate jacobian fibration associated to a cubic fourfold (\cite{LSV,Sacca} -- see also Section~\ref{section:LSV}). 

\subsection*{Acknowledgments}
We thank Alessio Bottini, Christian Lehn, Emanuele Macr{\`{\i}}, Giovanni Mongardi, Benedetta Piroddi and Antonio Rapagnetta for very interesting and important remarks at different stages of this work. We also wish to thank the two anonymous referees for the many questions and suggestions that greatly improved this work.
Franco Giovenzana was supported by the DFG through the research grant Le 3093/3-2.
Luca Giovenzana was supported by Engineering and Physical Sciences
Research Council (EPSRC) New Investigator Award EP/V005545/1 "Mirror
Symmetry for Fibrations and Degenerations".
Claudio Onorati was supported by the PRIN grant CUP E84I19000500006.

\thispagestyle{empty}

\section{Moduli spaces of sheaves on a K3 surface}\label{section:sheaves}

\subsection{Classical theory}
Let $S$ be a projective K3 surface and $v\in\Ktop(S)$ be a positive Mukai vector (\cite[Definition~0.1]{Yoshioka:ModuliSpaces}). 
For a choice of an ample divisor $H$, we consider the moduli space $M_v(S,H)$ of Gieseker--Maruyama $H$-semistable sheaves on $S$ of class $v$. When $v$ is primitive and $H$ is chosen generic with respect to $v$ (\cite[Section~2.1]{PR13}), the space $M_v(S,H)$ is a non-empty smooth and projective variety deformation equivalent to a Hilbert scheme of points of a K3 surface (\cite[Theorem~8.1]{Yoshioka:ModuliSpaces}). 
Here we are interested in the case $v$ is non primitive as described in the following theorem.
\begin{theorem}[\protect{\cite{O'Grady99}\cite[Th\'eor\`eme~1.1]{LehnSorger}}]
Suppose that $v=2w$, where $w$ is primitive and $w^2=2$. If $H$ is generic with respect to $v$, then $M_v(S,H)$ is a non-empty projective singular symplectic variety of dimension 10. Moreover, there exists a symplectic desingularisation 
$$ \pi\colon \widetilde{M}_v(S,H)\longrightarrow M_v(S,H), $$
where $\widetilde{M}_v(S,H)$ is a smooth and projective irreducible holomorphic symplectic variety. Moreover, $\widetilde{M}_v(S,H)$ is obtained by blowing up $M_v(S,H)$ at the singular locus with reduced scheme structure.
\end{theorem}

Any irreducible symplectic manifold that is deformation equivalent to $\widetilde{M}_v(S,H)$ is said to be of type OG10. Any singular symplectic variety that is locally trivially deformation equivalent to $M_v(S,H)$ is said to be singular of type OG10. (This makes sense since, by \cite[Theorem~1.6]{PR13}, the moduli spaces and the respective desingularisations obtained in the theorem are all deformation equivalent to each other.)

Recall that $\Ktop(S)$ is a unimodular lattice of rank $24$, where the pairing is given by the opposite of the Euler pairing. For $v\in\Ktop(S)$ we denote by $v^\perp$ the sublattice of $\Ktop(S)$ of vectors that are orthogonal to $v$. 
Recall also that $\Ktop(S)\cong\oH^{\operatorname{even}}(S,\Z)$, and it comes with a pure Hodge structure of weight 2 obtained by declaring 
\[
\oH^{\operatorname{even}}(S,\C)^{2,0}:=\oH^{2,0}(S),\qquad \oH^{\operatorname{even}}(S,\C)^{0,2}:=\oH^{0,2}(S) \]
and
\[\oH^{\operatorname{even}}(S,\C)^{1,1}:=\oH^{0}(S) \oplus\oH^{1,1}(S)\oplus \oH^4(S).
\]
In particular, $v^\perp$ inherits a Hodge structure of weight 2 as well. 

Notice also that the free $\Z$-module $\oH^{\operatorname{even}}(S,\Z)$ inherits a lattice structure from $\Ktop(S)$, and from now on we use the more common notation $\widetilde{\oH}(S,\Z)$
in order to explicit this lattice structure.

Finally, we say that two lattices, both with a Hodge structure, are \emph{Hodge-isometric} if there exists an isometry that is also an isomorphism of Hodge structures.

\begin{proposition}[\protect{\cite[Theorem~1.7]{PR13},\cite[Theorem~3.1]{PR:Factoriality}}]\label{prop:PR}
Let $S$, $H$ and $v$ be as in the theorem above, and let $\pi\colon \widetilde{M}_v(S,H)\to M_v(S,H)$ be the desingularisation morphism.
\begin{enumerate}
    \item The pullback 
    $$\pi^*\colon\oH^2(M_v(S,H),\Z)\longrightarrow\oH^2(\widetilde{M}_v(S,H),\Z)$$
    is injective. In particular, $\oH^2(M_v(S,H),\Z)$ has a pure Hodge structure of weight two and inherits a non-degenerate lattice structure. 
    \item The lattice $\oH^2(M_v(S,H),\Z)$ is Hodge-isometric to the lattice $v^\perp$.
    \item The Beauville--Bogomolov--Fujiki lattice $\oH^2(\widetilde{M}_v(S,H),\Z)$ is Hodge-isometric to the lattice
    \begin{equation}\label{eqn:Gamma} 
    \Gamma_v=\left\{\left(x,k\frac{\sigma}{2}\right)\in (v^\perp)^*\oplus\Z\frac{\sigma}{2}\mid k\in2\Z\Leftrightarrow x\in v^\perp\right\}, 
    \end{equation}
     where $\sigma^2=-6$ and corresponds to the class of the exceptional divisor of $\pi$. Here, the Hodge structure of $\Gamma_v$ is defined by the Hodge structure on $v^\perp$ and by declaring the class $\sigma$ to be of type $(1,1)$.
\end{enumerate}
\end{proposition}
\begin{remark}\label{rmk:H^2 singular}
Note that item (1) holds true for any singular symplectic variety by the work of Bakker and Lehn (see \cite[Lemma~2.1, Lemma~3.5]{BakkerLehn:Resolutions} and, more generally, \cite[Section~5]{BakkerLehn:Global}). 
\end{remark}

\subsection{Moduli of objects in the derived category}
As before, let $S$ be a projective K3 surface and let $\Db^b(S)$ be its derived category of coherent sheaves. By \cite[Theorem~1.2]{Bridgeland:Primo} there exists a complex manifold $\Stab(S)$ of Bridgeland stability conditions on $\Db^b(S)$ and, as customary, we denote by $\Stab^\dag(S)$ the distinguished connected component containing geometric stability conditions (\cite[Theorem~1.1]{Bridgeland}). If $v\in\Ktop(S)$ is a Mukai vector, then $\Stab^\dag(S)$ is decomposed in walls and chambers with respect to $v$, and we say that a stability condition is \emph{generic} if it belongs to one of the (open) chambers. By \cite[Theorem~2.15]{BM:MMP} and \cite[Theorem~1.3]{BM:Projectivity}, if $\tau\in\Stab^\dag(S)$ is generic, then there exists a non-empty coarse moduli space $M_v(S,\tau)$, parametrising S-equivalence classes of objects that are semistable with respect to $\tau$. Moreover, $M_v(S,\tau)$ is a normal projective and irreducible variety with $\Q$-factorial singularities, and if $v$ is primitive, then $M_v(S,\tau)$ is smooth and deformation equivalent to a Hilbert scheme of points on $S$.

\begin{theorem}[\protect{\cite[Proposition~2.2, Corollary~3.16]{MZ}}]
Let $\tau\in\Stab^\dag(S)$ be generic and $v=2w$, with $w$ primitive and $w^2=2$. Then $M_v(S,\tau)$ is singular, the singular locus being the locus of strictly semistable objects, and there exists a symplectic resolution of singularities
$$\pi\colon\widetilde{M}_v(S,\tau)\longrightarrow M_v(S,\tau).$$
Moreover $\widetilde{M}_v(S,\tau)$ is a projective irreducible holomorphic symplectic manifold of type OG10.
\end{theorem}
Under the genericity assumption of the stability condition, the singular locus of $M_v(S,\tau)$ is isomorphic to $\operatorname{Sym}^2M_w(S,\tau)$. (This classical result is usually referred to \cite[Lemma~1.1.5]{O'Grady99}, by using the tools developed in the proof of \cite[Theorem~1.3]{BM:Projectivity}.)

As before, let us describe the second cohomology groups of the singular and the smooth moduli spaces. The following result is due to Meachan and Zhang, we only prove the third statement, since it is not written anywhere, but all the tools are contained in \cite{MZ}.

\begin{proposition}[\protect{\cite[Theorem~2.7]{MZ}}]\label{prop:Gamma MZ}
Let $S$, $v=2w$ and $\tau$ be as above, and let $\pi\colon \widetilde{M}_v(S,\tau)\to M_v(S,\tau)$ be the desingularisation morphism.
\begin{enumerate}
    \item The pullback 
    $$\pi^*\colon\oH^2(M_v(S,\tau),\Z)\longrightarrow\oH^2(\widetilde{M}_v(S,\tau),\Z)$$
    is injective. In particular, $\oH^2(M_v(S,\tau),\Z)$ has a pure Hodge structure of weight two and inherits a non-degenerate lattice structure. 
    \item The Mukai morphism defines a Hodge isometry
    \[
    \theta_v\colon v^\perp \to \oH^2(M_v(S,\tau),\Z)
    \]
    that is invariant under Fourier--Mukai equivalences.
    \item The Beauville--Bogomolov--Fujiki lattice $\oH^2(\widetilde{M}_v(S,\tau),\Z)$ is Hodge-isometric to the lattice $\Gamma_v$ defined in (\ref{eqn:Gamma}).
\end{enumerate}
\end{proposition}
\begin{proof}
The only statement that is not contained in \cite[Theorem~2.7]{MZ} is the last one: its proof is essentially implicit in loc.\ cit., so we quickly recall the main ideas. By \cite[Lemma~7.3]{BM:Projectivity} the moduli space $M_v(S,\tau)$ is isomorphic to a moduli space $M_{\tilde{v}}(S',\alpha,H)$ of $H$-Gieseker semistable $\alpha$-twisted sheaves on a (possibly different) twisted K3 surface $S'$; the vector $\tilde{v}=2\tilde{w}$ is a Fourier--Mukai transform of $v=2w$. Therefore it is enough to prove the claim in this case. As recalled in the proof of \cite[Proposition~2.2]{MZ} (see \cite{Lieblich} and \cite{Yoshioka:TwistedSheaves}), moduli of twisted sheaves are constructed as GIT quotients and have the same deformation theory as untwisted sheaves. In particular $\widetilde{M}_{\tilde{v}}(S',\alpha,H)$ is constructed by blowing up the singular locus (which, under our genericity assumptions, is isomorphic to $\operatorname{Sym}^2M_{\tilde{w}}(S',\alpha,H)$) with its reduced scheme structure. We first claim that the exceptional divisor $\widetilde{\Sigma}_{\tilde{v}}(\alpha)$ of this resolution is an element of square $-6$ and divisibility $3$ in the Beauville--Bogomolov--Fujiki lattice $\oH^2(\widetilde{M}_{\tilde{v}}(S',\alpha,H),\Z)$. 
In order to prove this, we start by noticing that by \cite[Lemma~3.3, Proposition~3.7 and Proposition~3.9]{MZ} one can locally trivially deform $M_{\tilde{v}}(S',\alpha,H)$ to a moduli space of untwisted sheaves (notice that \cite[proposition~3.7]{MZ} holds conditionally to some technical assumptions, which we can always assume by performing a first general deformation of the twisted K3 surface).
More precisely, there exists a curve $B$ and a locally trivial family 
\[ p\colon\mathcal{M}\longrightarrow B \]
such that $\mathcal{M}_{b_1}=M_{\tilde{v}}(S,\alpha,H)$ and $\mathcal{M}_{b_2}=M_{v''}(S'',H'')$ for some polarised K3 surface $(S'',H'')$. If 
\begin{equation}\label{eqn:p tilde} \tilde{p}\colon\widetilde{\mathcal{M}}\longrightarrow B \end{equation}
is the induced family of the desingularisations (i.e.\ the relative blow-up of the singular loci), then by the proof of \cite[Proposition~2.16]{PR13} we get that the exceptional divisors of the fibres form a flat section $\tilde{\sigma}$ of the local system $R^2\tilde{p}_*\Z$. (Here and in the following all the local systems we consider come with a distinguished connection, the Gauss--Manin connection, and flatness of a section has to be interpreted with respect to this connection.) Therefore the claim follows from the untwisted case (\cite{Rapagnetta}).

Now define the homomorphism 
\begin{align*}
f_{\tilde{v}}(S',\alpha,H)\colon & \Gamma_{\tilde{v}}\longrightarrow\oH^2(\widetilde{M}_{\tilde{v}}(S,\alpha,H),\Z) \\
 &  \left(x,k\frac{\sigma}{2}\right) \mapsto \pi^*(\theta_{\tilde{v}}(x))+\frac{k}{2}\widetilde{\Sigma}_{\tilde{v}}(\alpha),
\end{align*}
 where $\pi\colon\widetilde{M}_{\tilde{v}}(S',\alpha,H)\to M_{\tilde{v}}(S',\alpha,H)$ is the desingularisation map and $\widetilde{\Sigma}_{\tilde{v}}(\alpha)$ is the class of the corresponding exceptional divisor. Let $\widetilde{p}\colon\widetilde{\mathcal{M}}\to B$ be a family as in (\ref{eqn:p tilde}). If $\widetilde{\Gamma}_{\tilde{v}}$ is the trivial local system on $B$ with stalk $\Gamma_{\tilde{v}}$, then $f_{\tilde{v}}(S',\alpha,H)$ extends to a morphism of local systems 
$$f_B\colon\widetilde{\Gamma}_{\tilde{v}}\longrightarrow R^2\tilde{p}_*\Z.$$
The proof is concluded as soon as $f_b$ is a Hodge-isometry for one point $b\in B$. By definition of $\tilde{p}\colon\widetilde{\mathcal{M}}\to B$, there exists $b\in B$ such that the fibre $\widetilde{\mathcal{M}}_b$ is isomorphic to a moduli space of untwisted sheaves, hence the claim follows from Proposition~\ref{prop:PR}.
\end{proof}

\begin{remark}\label{rmk:twisted}
The main tool in the results above is \cite[Lemma~7.3]{BM:Projectivity}, which translates problems on $M_v(S,\tau)$ to problems on $M_{\tilde{v}}(S,\alpha)$. Hidden in the proof of Proposition~\ref{prop:Gamma MZ} (more precisely in the proof of \cite[Proposition~2.2]{MZ}) there is the statement that everything we said so far holds for moduli spaces of twisted sheaves on a K3 surface. (For primitive Mukai vectors, the analogous statement is \cite[Theorem~6.10]{BM:Projectivity}.) In particular, for a generic choice of the polarisation, $M_{\tilde{v}}(S,\alpha)$ admits a symplectic desingularisation $\widetilde{M}_{\tilde{v}}(S,\alpha)$ obtained by blowing up the singular locus (identified with $\Sym^2 M_{\tilde{v}/2}(S,\alpha)$) with its reduced scheme structure. Moreover, $\widetilde{M}_{\tilde{v}}(S,\alpha)$ is an irreducible holomorphic symplectic manifold, $\oH^2(M_{\tilde{v}}(S,\alpha),\Z)$ is Hodge-isometric to $\tilde{v}^\perp$ and $\oH^2(\widetilde{M}_{\tilde{v}}(S,\alpha),\Z)$ is Hodge-isometric to $\Gamma_{\tilde{v}}$.
We will use this remark later in the proof of Theorem~\ref{thm:main}.
\end{remark}

\section{LPZ varieties}\label{section:BBF LPZ}

Let $V$ be a smooth cubic fourfold and $\mathcal{A}_V$ be the Kuznetsov component defined by the semi-orthogonal decomposition
\begin{equation}\label{eqn:D(V)}
\operatorname{D}^b(V)=\langle \mathcal{A}_V,\mathcal{O}_V,\mathcal{O}_V(1),\mathcal{O}_V(2)\rangle. 
\end{equation}
The category $\mathcal{A}_V$ is a $\operatorname{CY}_2$-category (\cite[Corollary~4.3]{Kuzi:V14}; see also \cite{Kuzi:Fractional} for a general account about Calabi--Yau categories).
The Mukai lattice $\widetilde{\operatorname{H}}(\mathcal{A}_V)$ introduced in \cite[Definition~2.2]{Add-Thomas} is defined as
\[
\widetilde{\operatorname{H}}(\mathcal{A}_V):= \langle [\mathcal{O}_V],[\mathcal{O}_V(1)],[\mathcal{O}_V(2)]\rangle^\perp \subset \Ktop(V).
\]
Here $\Ktop(V)$ is the topological K-theory of $V$ equipped with the Euler pairing. The Mukai lattice is equipped with a pure Hodge structure of weight 2 induced by the Hodge structure on $\oH^*(V,\Z)$. 
More precisely, if $v\colon\Ktop(V)\to\oH^*(V,\Q)$ is the morphism associating to a sheaf $F$ the vector $v(F)=\operatorname{ch}(F)\sqrt{\operatorname{td}_V}$, then 
\[\widetilde{\oH}(\mathcal{A}_V)^{2,0}=v^{-1}\left( \oH^{3,1}(V)\right). \]
As an abstract lattice, $\widetilde{\operatorname{H}}(\mathcal{A}_V)$ is isometric to $U^{\oplus4}\oplus E_8(-1)^{\oplus2}$, where $U$ is the unimodular hyperbolic plane and $E_8(-1)$ is the negative definite lattice associated to the Dynkin diagram $E_8$ (\cite[Sec.2.3]{Add-Thomas}). 
Notice that if $\mathcal{A}_V\cong\operatorname{D}^b(S)$, for some K3 surfaces $S$, then there is a Hodge-isometry
\[ \widetilde{\operatorname{H}}(\mathcal{A}_V)\cong\widetilde{\oH}(S,\Z). \]

If $\operatorname{pr}\colon\operatorname{D}^b(V)\to\mathcal{A}_V$ denotes the projection functor with respect to the decomposition (\ref{eqn:D(V)}), we define the elements $\lambda_1,\lambda_2\in\widetilde{\operatorname{H}}(\mathcal{A}_V)$ as the classes of $\operatorname{pr}(\mathcal{O}_l(1))$ and $\operatorname{pr}(\mathcal{O}_l(2))$, respectively. Here $l\subset V$ is a line. The classes $\lambda_1$ and $\lambda_2$ are algebraic, i.e.\ $\lambda_1,\lambda_2\in \widetilde{\operatorname{H}}^{1,1}(\mathcal{A}_V)$, and they generate a lattice isometric to $A_2$, the rank $2$ lattice associated to the Dynking diagram $A_2$. More precisely $\lambda_i^2=2$ and $(\lambda_1,\lambda_2)=-1$ (see \cite[(2.5)]{Add-Thomas}). 
Notice that the lattice $A_2$ is always contained in the algebraic part of $\widetilde{\operatorname{H}}(\mathcal{A}_V)$ by construction, and they coincide for the very general cubic fourfold. Moreover, the orthogonal complement $A_2^\perp$ in $\widetilde{\operatorname{H}}(\mathcal{A}_V)$ is Hodge isometric to the primitive cohomology group $\oH^4(V,\Z)_{\operatorname{prim}}$ (see \cite[Proposition~2.3]{Add-Thomas}).

Stability conditions on $\mathcal{A}_V$ have been constructed in \cite[Theorem~1.2]{BLMS}.
Let $\lambda\in \widetilde{\operatorname{H}}(\mathcal{A}_V)^{1,1}$ be a Mukai vector and let us suppose that $\lambda=2\lambda_0$, where $\lambda_0$ is primitive and $\lambda_0^2=2$. For example we can take $\lambda_0=\lambda_1+\lambda_2$. 
For a generic stability condition $\tau$, we consider the moduli stack $\underline{M}_\lambda(V,\tau)$ of $\tau$-semistable objects in $\mathcal{A}_V$. Here, generic means that any strictly  semistable object of class $\lambda$ is S-equivalent to the direct sum of two stable objects of class $\lambda_0$. The moduli stack $\underline{M}_\lambda(V,\tau)$ has a good moduli space $M_\lambda(V,\tau)$ that exists as a proper algebraic space (\cite[Theorem~21.24]{BLNPS}.
\begin{theorem}[\protect{\cite[Theorem~1.1]{PertusiCo}}]
Under the hypothesis above, there exists a smooth and projective variety $\widetilde{M}_\lambda(V,\tau)$ and a symplectic resolution 
$$\pi\colon\widetilde{M}_\lambda(V,\tau)\to M_\lambda(V,\tau).$$ 
Moreover, $\widetilde{M}_\lambda(V,\tau)$ is an irreducible holomorphic symplectic manifold of type OG10.
\end{theorem}
As in the classical case, the singular locus of $M_\lambda(V,\tau)$ is identified with the symmetric product $\operatorname{Sym}^2M_{\lambda_0}(V,\tau)$, and the morphism $\pi$ is the blow-up of $M_\lambda(V,\tau)$ at the singular locus (with its reduced scheme structure).

Let us recall the main features of $M_\lambda(V,\tau)$ and $\widetilde{M}_\lambda(V,\tau)$.
\begin{lemma}
$M_\lambda(V,\tau)$ is a normal and projective variety.
\end{lemma}
\proof
This is explained in the proof of \cite[Theorem~3.1]{PertusiCo} (see \cite[Section~3.7]{PertusiCo}). First of all, Li, Pertusi and Zhao prove that there exists an ample line bundle $\mathcal{L}$ on $M_\lambda(V,\tau)$, giving an embedding of $M_\lambda(V,\tau)$ into a projective space. By \cite[\href{https://stacks.math.columbia.edu/tag/0D2W}{Lemma 0D2W}]{Stack}, one concludes that $M_\lambda(V,\tau)$ is a scheme. Now, as explained in \cite[Remark~3.6]{PertusiCo}, $M_\lambda(V,\tau)$ has a local description in terms of nilpotent orbits as in the classical case, from which it follows at once both the normality and the fact that $M_\lambda(V,\tau)$ is a variety.
\endproof
\begin{lemma}
$M_\lambda(V,\tau)$ is a singular symplectic variety, in particular it has rational singularities.
\end{lemma}
\proof
This follows from the fact that $M_\lambda(V,\tau)$ admits an irreducible symplectic resolution. Thus, it has canonical singularities, hence rational singularities \cite{Elkik} (see also \cite[Proposition~1.3]{Beauville:Singular}).
\endproof
\begin{lemma}\label{lemma: pi injective}
$\operatorname{H}^i(M_\lambda(V,\tau),\mathbb{Z})$ has a pure Hodge structure of weight $i$, for $i=1,2$. Moreover, the pullback map 
$$\pi^*\colon\operatorname{H}^1(M_\lambda(V,\tau),\mathbb{Z})\to\operatorname{H}^1(\widetilde{M}_\lambda(V,\tau),\mathbb{Z})$$ 
is an isomorphism and the pullback map 
$$\pi^*\colon\operatorname{H}^2(M_\lambda(V,\tau),\mathbb{Z})\to\operatorname{H}^2(\widetilde{M}_\lambda(V,\tau),\mathbb{Z})$$ 
is injective.

In particular, $\operatorname{H}^2(M_\lambda(V,\tau),\mathbb{Z})$ is endowed with a non-degenerate symmetric bilinear form induced by the one on $\widetilde{M}_\lambda(V,\tau)$ via $\pi^*$, turning it into a lattice of signature $(3,b_2(M_\lambda(V,\tau))-3)$.
\end{lemma}
\proof
As we already noticed in Remark~\ref{rmk:H^2 singular}, since $M_\lambda(V,\tau)$ is a projective variety with rational singularities, this is \cite[Lemma~2.1, Lemma~3.5]{BakkerLehn:Resolutions}.
\endproof
\begin{remark}
Since $M_\lambda(V,\tau)$ is a singular symplectic variety, $\operatorname{H}^2(M_\lambda(V,\tau),\mathbb{Z})$ has an intrinsic non-degenerate symmetric bilinear form turning it into a lattice by \cite[\S 5]{BakkerLehn:Global} and references therein. 
This intrinsic lattice structure coincides with the one induced by the irreducible symplectic desingularisation as in the Lemma~\ref{lemma: pi injective}.

\end{remark}

\begin{example}[The vector $\lambda=2\lambda_1+2\lambda_2$]\label{example:LPZ}
Let $V$ be a smooth cubic fourfold and $Y\subset V$ a smooth linear section. Let $E$ be a rank 2 vector bundle on $Y$ with trivial first Chern class and second Chern class of degree $2$. This is called an instanton bundle of charge $2$ on a smooth cubic threefold, see \cite{Druel} for a detailed study of the corresponding moduli space. 
Let $i\colon Y\to V$ be the closed embedding and set $F=i_*E$; it is easy to see that $F\in\mathcal{A}_V$. Li, Pertusi and Zhao prove in \cite{PertusiCo} that there exists a stability condition $\bar{\tau}$ such that $F$ is $\bar{\tau}$-stable (this stability condition is the one constructed in \cite{BLMS} when $V$ is very general). Moreover, by a direct computation one can see that the Mukai vector of $F$ is $\lambda=2\lambda_1+2\lambda_2$. The stability condition $\bar{\tau}$ is generic with respect to $\lambda$ and  the moduli space $M_\lambda(V,\bar{\tau})$ admits a desingularisation $\widetilde{M}_\lambda(V,\bar{\tau})$ that is an irreducible holomorphic symplectic manifold of type OG10. Moreover, in \cite[Section~6]{PertusiCo} the authors construct a birational lagrangian fibration structure on $\widetilde{M}_\lambda(V,\bar{\tau})$ and show that $\widetilde{M}_\lambda(V,\bar{\tau})$ is in fact birational to the twisted intermediate jacobian fibration constructed in \cite{Voisin:Twisted} (cf.\ Section~\ref{section:LSV}).

The generality assumption can conjecturely be made more precise by saying that $V$ does not contain a plane or a rational cubic scroll (see \cite[Section~5.7]{PertusiCo} for the case of a rational cubic scroll -- the case of a plane is expected to behave similarly), and in these two cases one expects to find examples of the walls of the K\"ahler cone described in \cite{MZ} (in the singular setting) and \cite{MO} (in the desingularisation). 
\end{example}

Denote by $M^s_\lambda(V,\tau)\subset M_\lambda(V,\tau)$ the open subset consisting of stable objects and take a quasi-universal family $F\in\operatorname{D}(M^s_\lambda(V,\tau)\times V)_{\operatorname{perf}}$ of similitude $\rho$ (see \cite[Definition~4.5, Remark~4.6]{BM:Projectivity}).
Consider the induced map
$$\theta'\colon\widetilde{\operatorname{H}}(\mathcal{A}_V)\to\operatorname{H}^2(M^s_\lambda(V,\tau),\mathbb{Q}),\qquad x\mapsto\frac{1}{\rho} \left[p_{M^s*}\left(\operatorname{ch}(F). p_{V}^*\left(x^\vee\sqrt{\operatorname{td}_V}\right)\right)\right]_1,$$
where $p_V\colon V\times M^s_\lambda(V,\tau) \to V$ and $p_{M^s}\colon V\times M^s_\lambda(V,\tau) \to M^s_\lambda(V,\tau)$ are the projections and $x^\vee$ is the dual class to $x$ in $\Ktop(V)$.
In the following we work with the restriction
$$ \theta'_\lambda\colon\lambda^\perp\to\operatorname{H}^2(M^s_\lambda(V,\tau),\mathbb{Q}),$$
which do not depend on the choice of the quasi-universal family.
Our first remark is that this morphism extends to a morphism to $\operatorname{H}^2(M_\lambda(V,\tau),\Z)$.
\begin{lemma}
The following pullback morphism is an isomorphism: $$i^*\colon\oH^2(M_\lambda(V,\tau),\Q)\to\oH^2(M^s_\lambda(V,\tau),\Q).$$
In particular, there exists a canonical morphism $\theta_{\lambda}$ making the diagram
\begin{equation*}
\xymatrix{
\lambda^\perp  \ar[dr]_{\theta'_{\lambda}} \ar[rr]^{\theta_{\lambda}} & & \operatorname{H}^2(M_\lambda(V,\tau),\mathbb{Q}) \ar[dl]^{i^*} \\
 & \operatorname{H}^2(M^s_\lambda(V,\tau),\mathbb{Q}) &
}
\end{equation*}
commutative.
\end{lemma}
\begin{proof}
First of all, notice that the claim is topological in nature, and that the variety $M_\lambda(V,\tau)$ is a locally trivial deformation of a singular moduli space $M_{v}(S,H)$, where $S$ is a K3 surface, $v=2w$ with $w^2=2$ and $H$ is generic with respect to $v$. Therefore it is enough to prove the claim for $M=M_v(S,H)$. 
This is essentially done in \cite[Lemma~3.7]{PR13} (see also the second paragraph of \cite[page~18]{PR13}), let us briefly explain why. 
In the proof of \cite[Lemma~3.7]{PR13}, Perego and Rapagnetta prove that there exists a commutative diagram with exact rows
\[  
\xymatrix{
0\ar@{->}[r]\ar@{->}[d] & \operatorname{H}^2(M,\Z)\ar@{->}[r]\ar@{->}[r]^{i^*}\ar@{^{(}->}[d]^{\pi^*} & \operatorname{H}^2(M^s,\Z)\ar@{->}[d]^{\cong} \\
\Z\ar@{^{(}->}[r]^-{\tilde{c}} & \operatorname{H}^2(\widetilde{M},\Z)\ar@{->>}[r]^-{\tilde{\i}} & \operatorname{H}^2(\pi^{-1}(M^s),\Z),
}
\]
where $\pi\colon\widetilde{M}\to M$ is the symplectic resolution and $\tilde{c}(1)=\widetilde{\Sigma}$ is the exceptional divisor. Notice that the image of $\tilde{c}$ is not contained in $\pi^*(\operatorname{H}^2(M,\Z))$, so that eventually the defect of surjectivity of $i^*$ is contained in the torsion part of the quotient $\oH^2(\widetilde{M},\Z)/\oH^2(M,\Z)$. Since the latter is a finite group, the claim follows.

\end{proof}
\begin{proposition}\label{prop:v perp}
The homomorphism $\theta_{\lambda}$ is integral, i.e.\ 
\[
\theta_\lambda(\lambda^\perp)\subseteq\oH^2(M_\lambda(V,\tau),\Z).
\]
Moreover, 
\[
\theta_\lambda\colon\lambda^\perp\longrightarrow \oH^2(M_\lambda(V,\tau),\Z)
\]
is a Hodge-isometry.
\end{proposition}
\proof
The claims follow at once via a deformation argument.
Let $\mathcal{V}\to (B, 0)$ be a deformation of $V$ over a curve $B$ such that for a point $b\in B$ the cubic fourfold $\mathcal{V}_{b}$ has an associated K3 surface $S$ with $D^b(S)\simeq \mathcal{A}_{V_{b}}$.
 By \cite[Proposition~3.7]{PertusiCo} the relative moduli space 
 \[
p\colon\mathcal{M}\longrightarrow B.
\]
exists and is a locally trivial deformation of $\mathcal M _0 = M_\lambda(V,\tau)$ with $\mathcal M_b$ isomorphic to $M_\lambda(S, \tau)$ for some stability condition $\tau$ on $D^b(S)$.

Since $\theta'_{\lambda,\Q}$ is defined via a quasi-universal family, it also deforms: there exists a well-defined morphism of local systems
\[
\tilde{\theta}\colon\lambda^\perp_{\Q}\longrightarrow R^2p_*\Q.
\] 
Now, over the point $b$ the map $\tilde{\theta}_{b}$ coincides with the map $\theta_\lambda\colon\lambda^\perp\to\oH^2(M_\lambda(S,\tau),\Q)$ which is integral and a Hodge-isometry by item (2) of Proposition~\ref{prop:Gamma MZ} (more precisely, \cite[Theorem~2.7]{MZ}). Since $B$ is connected, the same must be true for any other point of $B$, in particular for $0$.
\endproof
\begin{remark}
If the Mukai vector $v$ is primitive, the same result is \cite[Theorem~29.2]{BLNPS}.
\end{remark}

\begin{corollary}\label{cor:Pic}
Under the hypotheses above, one has that 
\[
\Pic(M_\lambda(V,\tau))=\oH^{1,1}(M_\lambda(V,\tau),\Z)=\oH^{1,1}(M_\lambda(V,\tau))\cap\oH^2(M_\lambda(V,\tau),\Z)\cong(\lambda^\perp)^{1,1}.
\]
\end{corollary}
\begin{proof}
Straightforward (see \cite[Corollary~2.8]{MZ} for the analogous result in the commutative case).
\end{proof}

\begin{corollary}\label{cor:Donaldson}
Suppose that $\lambda\in A_2$. Then there exists an isometric embedding of Hodge structures
\begin{equation}\label{eqn:Donaldson}
\xi_\lambda\colon\oH^4(V,\Z)_{\operatorname{prim}}\longrightarrow \oH^2(M_\lambda(V,\tau),\Z).
\end{equation}
\end{corollary}
\begin{proof}
Since $\lambda\in A_2\subset\widetilde{\oH}(\mathcal{A}_V,\Z)$ and $A_2^\perp=\oH^4(V,\Z)_{\operatorname{prim}}$ by \cite[Proposition~2.3]{Add-Thomas}, we have a natural inclusion 
\[
\oH^4(V,\Z)_{\operatorname{prim}}\subset \lambda^\perp
\]
that preserves both the lattice and the Hodge structures. Then $\xi_\lambda$ is defined as the restriction of $\theta_\lambda$ and the claim follows from Proposition~\ref{prop:v perp}.
\end{proof}
\begin{remark}[The Donaldson morphism]
Let $S$ be a projective K3 surface and $v=(2,0,-2)$. The moduli space $M_v(S,H)$ is the singular variety studied in the original paper \cite{O'Grady99}. For what follows, we refer to the book \cite{FM}. A general point in $M_v(S,H)$ corresponds to a slope-stable rank 2 vector bundle $E$ on $S$ whose first Chern class vanishes and whose second Chern class has degree $4$. The underlying complex vector bundle of $E$ has the structure of a $\operatorname{SU}(2)$-principal bundle and it admits an anti-self-dual connection of charge $0$. In fact, more precisely, the open locus $M^{\operatorname{lf}}_v(S,H)$ of locally free sheaves is isomorphic (as a real analytic space) to the corresponding moduli space of anti-self-dual connections (\cite[Theorem IV.3.9]{FM}). By Uhlenbeck's Weak Compactness Theorem (cf.\ \cite[Theorem~III.3.15]{FM}), the space of anti-self-dual connections admits a natural compactification that we denote by $\overline{M}^U_S$. By \cite[Corollary~4.3]{Li}, the Uhlenbeck space $\overline{M}^U_S$ has an algebraic structure and there exists a divisorial contraction $\phi\colon M_{(2,0,-2)}(S,H)\to \overline{M}^U_S$ (see also \cite[Section~3.1]{O'Grady99}). The \emph{Donaldson morphism} is the morphism 
\[\delta\colon\oH^2(S,\Z)\longrightarrow\oH^2(\overline{M}^U_S,\Z) \] 
defined by taking the slant product with a universal bundle (see \cite[Theorem III.3.10, Theorem III.6.1]{FM}). Notice that it is injective (cf.\ \cite[Proposition VII.2.17]{FM}). Pulling back via $\phi$ gives an injective morphism
\[ \xi_S=\phi^*\circ\delta\colon\oH^2(S,\Z)\longrightarrow\oH^2(M_{(2,0,-2)}(S,H),\Z). \]

If $U\subset\oH^{\operatorname{even}}(S,\Z)$ is the hyperbolic plane generated by $\oH^0(S,\Z)$ and $\oH^4(S,\Z)$, then $v=(2,0,-2)\in U$ and $\oH^2(S,\Z)=U^\perp\subset v^\perp$. By definition, the Donaldson morphism coincides with the restriction to $\oH^2(S,\Z)$ of the morphism $v^\perp\to\oH^2(M_v(S,H),\Z)$ (see Proposition~\ref{prop:PR}).

Now, by definition $\lambda\in A_2\subset\widetilde{\oH}(\mathcal{A}_V)$ and by \cite[Proposition~2.3]{Add-Thomas} $A_2^\perp=\oH^4(V,\Z)_{\operatorname{prim}}$. Moreover, when $V$ has an associated K3 surface, then there is a Hodge-isometric embedding $\oH^2(S,\Z)_{\operatorname{prim}}\to\oH^4(V,\Z)_{\operatorname{prim}}$ (\cite[Proposition~1.25]{Huybrechts:Gargnano}).

Therefore we regard $\xi_\lambda$ in Corollary~\ref{cor:Donaldson} as a generalised version of the Donaldson morphism.

\end{remark}

\begin{example}\label{example:Pic}
Let $V$ be a smooth cubic fourfold and consider the LPZ variety in Example~\ref{example:LPZ}. 
There is a Hodge-isometric embedding
\begin{equation}\label{eqn:isometric embedding} \oH^4(V,\Z)_{\operatorname{prim}}\longrightarrow\oH^2(\widetilde{M}_{2\lambda_1+2\lambda_2}(V,\tau),\Z) 
\end{equation}
obtained by composing the Donaldson morphism (\ref{eqn:Donaldson}) with the pullback by the desingularisation map (see Lemma~\ref{lemma: pi injective}). We claim that the orthogonal complement of $\oH^4(V,\Z)_{\operatorname{prim}}$ in $\oH^2(\widetilde{M}_{2\lambda_1+2\lambda_2}(V,\tau),\Z)$ is generated by the class of the exceptional divisor and by an algebraic and isotropic class. In fact, by \cite[Theorem~1.3]{PertusiCo}, $\widetilde{M}_{2\lambda_1+2\lambda_2}(V,\tau)$ is always birational to an irreducible holomorphic symplectic manifold having a lagrangian fibration structure. If $b_V\in\Pic(\widetilde{M}_{2\lambda_1+2\lambda_2}(V,\tau))$ is the isotropic movable class corresponding to the pullback of the polarisation on the base of the fibration, then $b_V$ remains of type $(1,1)$ on all the deformations induced by deformations of the cubic fourfolds. The same holds for the class of the exceptional divisor.
Therefore the claim follows by deforming to a very general cubic fourfold in the sense of Hassett.

If we denote by $\widetilde{\Sigma}_V$ the class of the exceptional divisor, then 
\[ \Pic(\widetilde{M}_{2\lambda_1+2\lambda_2}(V,\tau))=\langle \oH^{2,2}(V,\Z)_{\operatorname{prim}}, \widetilde{\Sigma}_V,b_V\rangle, \]
where $\oH^{2,2}(V,\Z)_{\operatorname{prim}}=\oH^4(V,\Z)_{\operatorname{prim}}\cap\oH^{2,2}(V).$ 

We finish by noticing that the lattice generated by $\widetilde{\Sigma}_V$ and $b_V$ is isometric to the non-unimodular hyperbolic plane $U(3)$. In fact, if the cubic fourfold is very general, then by \cite[Theorem~1.3]{PertusiCo}
we have a chain of equalities 
\[ \langle \widetilde{\Sigma}_V,b_V\rangle=\Pic(\widetilde{M}_{2\lambda_1+2\lambda_2}(V,\tau))\cong\Pic(\operatorname{IJ}^t(V))=U(3), \]
where the last equality is \cite[Lemma~6.2]{MO}. Here $\operatorname{IJ}^t(V)$ is the symplectic compactification of the twisted intermediate jacobian fibration of $V$ \cite{Voisin:Twisted} (see also Section~\ref{section:LSV}).

Finally, let us also notice that the lattice generated by $\widetilde{\Sigma}_V$ and $b_V$ is primitively embedded in $\Pic(\widetilde{M}_{2\lambda_1+2\lambda_2}(V,\tau))$.
\end{example}

As a consequence of Corollary~\ref{cor:Donaldson}, we get the following Torelli-like statement for certain LPZ varieties, namely the varieties $M_{2\lambda_1+2\lambda_2}(V,\tau)$ in Example~\ref{example:LPZ}. As usual, we denote by $\widetilde{M}_{2\lambda_1+2\lambda_2}(V,\tau)$ the symplectic desingularisation, and we recall that $\widetilde{M}_{2\lambda_1+2\lambda_2}(V,\tau)$ has a birational lagrangian fibration structure, i.e.\ it is birational to an irreducible holomorphic symplectic manifold having a lagrangian fibration structure (see Section~\ref{section:LSV} or \cite[Theorem~1.3]{PertusiCo}).
\begin{theorem}\label{thm:Torelli}
Let $V_1$ and $V_2$ be two smooth cubic fourfolds and consider the desingularised LPZ varieties $\widetilde{M}_{2\lambda_1+2\lambda_2}(V_1,\tau_1)$ and $\widetilde{M}_{2\lambda_1+2\lambda_2}(V_2,\tau_2)$, where $\tau_i$ is a $2\lambda_1+2\lambda_2$-generic stability condition on the Kuznetsov component $\mathcal{A}_{V_i}$ of $V_i$.
The following conditions are equivalent.
\begin{enumerate}
   \item $\widetilde{M}_{2\lambda_1+2\lambda_2}(V_1,\tau_1)$ is birational to $\widetilde{M}_{2\lambda_1+2\lambda_2}(V_2,\tau_2)$ such that:
   \begin{itemize}
       \item the birationality preserves the exceptional divisors; 
       \item the birationality commutes with the birational lagrangian fibration structures.
   \end{itemize}
   \item $V_1$ and $V_2$ are isomorphic.
 \end{enumerate}
\end{theorem}
\begin{proof}
Given a birational map 
\[ \widetilde{M}_\lambda(V_1,\tau_1)\longrightarrow \widetilde{M}_\lambda(V_2,\tau_2),\]
there is an induced Hodge-isometry 
\begin{equation}\label{eqn:p}
\oH^2(\widetilde{M}_\lambda(V_1,\tau_1),\Z)\longrightarrow\oH^2(\widetilde{M}_\lambda(V_2,\tau_2),\Z). 
\end{equation}
Now, as we saw in Example~\ref{example:Pic}, there is a primitive embedding
\[ \oH^4(V_i,\Z)_{\operatorname{prim}}\hookrightarrow\oH^2(\widetilde{M}_{2\lambda_1+2\lambda_2}(V_i,\tau_i),\Z), \]
and the orthogonal complement of $\oH^4(V_i,\Z)_{\operatorname{prim}}$ in $\oH^2(\widetilde{M}_{2\lambda_1+2\lambda_2}(V_i,\tau_i),\Z)$ is generated by the exceptional divisor and an algebraic isotropic class corresponding to the class of the birational lagrangian fibration structure.

Since by hypothesis these two classes are preserved, the Hodge-isometry (\ref{eqn:p}) restricts to a Hodge-isometry
\[ \oH^4(V_1,\Z)_{\operatorname{prim}}\longrightarrow \oH^4(V_2,\Z)_{\operatorname{prim}} \]
and, by the Torelli theorem for cubic fourfolds (\cite{Voisin:TorelliCubics}), $V_1$ and $V_2$ are isomorphic.
\end{proof}

Finally, we give a lattice-theoretic description of the period of the smooth varieties $\widetilde{M}_\lambda(V,\tau)$.
As in the commutative case, we define the lattice 
\begin{equation}\label{eqn:Gamma V}
\Gamma_\lambda=\left\{\left(x,k\frac{\sigma}{2}\right)\in (\lambda^\perp)^*\oplus\Z\frac{\sigma}{2}\mid k\in2\Z\Leftrightarrow x\in \lambda^\perp\right\}.
\end{equation}
This comes with a natural Hodge structure given by the Hodge structure on $(\lambda^\perp)^*$ and by declaring $\sigma$ to be of type $(1,1)$.
\begin{proposition}\label{prop:Gamma}
The morphism 
$$
f_V\colon \Gamma_\lambda\longrightarrow\operatorname{H}^2(\widetilde{M}_\lambda(V,\tau),\mathbb{Z}),\ 
\left(x,k\frac{\sigma}{2}\right) \mapsto \pi_V^*(\theta_\lambda(x))+\frac{k}{2}\widetilde{\Sigma}_\lambda
$$
where $\widetilde{\Sigma}_\lambda$ is the class of the exceptional divisor, is a Hodge-isometry.
\end{proposition}
\proof
The proof is as in Proposition~\ref{prop:Gamma MZ}.
By definition, $f_V$ respects the Hodge structures. Let now $p\colon\mathcal{M}\to B$ be a family of singular LPZ varieties induced by a family of cubic fourfolds as in \cite[Proposition~3.7]{PertusiCo}, and $\tilde{p}\colon\widetilde{\mathcal{M}}\to B$ the associated family of LPZ manifolds.
If $\widetilde{\Gamma}_\lambda$ denotes the trivial local system on $B$ with stalk $\Gamma_\lambda$, then the morphism $f_V$ extends to a morphism of local systems
$$\widetilde{\Gamma}_\lambda\longrightarrow R^2\tilde{p}_*\mathbb{Z}.$$
Therefore, it is enough to prove the claim for a point of $B$. If we choose the family $B$ such that there exists a point $b\in B$ such that the corresponding cubic fourfold is Pfaffian, then the result follows from Proposition~\ref{prop:Gamma MZ}.
\endproof

Recall that a locally factorial variety is a variety such that any Weil divisor is Cartier. If $m$ is an integer, then a variety is $m$-factorial if for any Weil divisor $D$ there exists $k\leq m$ such that $kD$ is Cartier.

\begin{corollary}\label{cor:factoriality}
The moduli space $M_{\lambda}(V,\tau)$ is either locally factorial or $2$-factorial. More precisely, if we write $\lambda=2\lambda_0$:
\begin{itemize}
    \item $M_{\lambda}(V,\tau)$ is locally factorial if and only if $(\lambda_0,u)\in 2\Z$ for every $u\in\widetilde{\oH}(\mathcal{A}_V)^{1,1}$;
    \item $M_{\lambda}(V,\tau)$ is $2$-factorial if and only if there exists $u\in\widetilde{\oH}(\mathcal{A}_V)^{1,1}$ such that $(\lambda_0,u)=1$.
\end{itemize}
\end{corollary}
\begin{proof}
The proof is the same as the proof of \cite[Theorem~1.1]{PR:Factoriality} (see \cite[Section~4.1]{PR:Factoriality}), using our Proposition~\ref{prop:v perp} and Proposition~\ref{prop:Gamma} instead of \cite[Theorem~1.7]{PR13} and \cite[Theorem~3.1]{PR:Factoriality}, respectively. Let us recall here the main steps for the reader's convenience. 

First of all, if $A^1(M_\lambda(V,\tau))$ denotes the Weil class group, then we need to compute the quotient $A^1(M_\lambda(V,\tau))/\Pic(M_\lambda(V,\tau))$. Now, by Corollary~\ref{cor:Pic}, we have that $\Pic(M_\lambda(V,\tau))\cong(\lambda^\perp)^{1,1}$. On the other hand, since the singularities of $M_\lambda(V,\tau)$ are in codimension $2$, we have that $A^1(M_\lambda(V,\tau))\cong\Pic(M^s_\lambda(V,\tau))$. There is a short exact sequence
\[
0\to\Z\to\Pic(\widetilde{M}_\lambda(V,\tau))\to\Pic(M^s_\lambda(V,\tau))\to0,
\]
where the first map is defined by mapping $1$ to the class $\widetilde{\Sigma}_\lambda$ of the exceptional divisor and the second map is the restriction. Combining this with Proposition~\ref{prop:Gamma}, we get 
\[ 
A^1(M_\lambda(V,\tau))\cong \frac{\Gamma_\lambda^{1,1}}{\Z\sigma},
\]
and eventually
\[
\frac{A^1(M_\lambda(V,\tau))}{\Pic(M_\lambda(V,\tau))}\cong\frac{\Gamma_\lambda^{1,1}}{(\lambda^\perp)^{1,1}\oplus\Z\sigma}.
\]
The proof is now reduced to a lattice-theoretic computation.
\end{proof}

\begin{example}
If $\lambda_0=\lambda_1+\lambda_2$, then $(\lambda_0,\lambda_1)=1$. Therefore the moduli space $M_{2\lambda_1+2\lambda_2}(V,\tau)$ of Example~\ref{example:LPZ} is $2$-factorial.
\end{example}

\section{When is an LPZ variety birational to a moduli space of sheaves?}
 
In the following we denote by $\mathcal{C}$ the moduli space of smooth cubic fourfolds and by $\mathcal{C}_d$ the irreducible (Hassett) divisor consisting of special cubic fourfolds of discriminant $d$. Recall that $\mathcal{C}_d$ is non-empty if and only if $d>6$ and $d\equiv 0,2\pmod{6}$ (see \cite[Theorem~4.3.1]{Hassett}). 

We consider the following two properties for $d$
\begin{center}
    $(**)$: $d$ divides $2n^2 + 2n +2$ for some $n\in \Z$;\\
    $(**')$: in the prime factorization of $d/2$, primes $p\equiv 2 (3)$ appear with even exponents.
\end{center}

As it has been well summarized in  \cite[Proposition~1.13, Proposition~1.24, Corollary~1.26]{Huybrechts:Gargnano}, from work of Hassett, Addington, Addington and Thomas, and Huybrechts \cite{Addington:Fano, Add-Thomas, Hassett,huybrechts-twisted}, the first condition is equivalent to the existence of an associated K3 surface; the second condition is equivalent to the existence of an associated twisted K3 surface.

A birational map $f\colon X\dashrightarrow Y$ between singular LPZ varieties is called \emph{stratum preserving} if it is defined at the generic point of the singular locus of $X$ and maps it to the generic point of the singular locus of $Y$. Since LPZ varieties admit an irreducible symplectic desingularisation, this is equivalent to requiring that there exists a birational map between the desingularisations preserving the exceptional divisors.

\begin{proposition}\label{thm:main}
For any smooth cubic fourfold $V$, we consider the singular LPZ variety $M_\lambda(V,\tau)$, where $\lambda=2\lambda_0$, $\lambda_0^2=2$ and $\tau$ is a generic stability condition on the Kuznetsov component $\mathcal{A}_V$.
\begin{enumerate}
\item The following conditions are equivalent:
 \begin{enumerate}
   \item $M_\lambda(V,\tau)$ is stratum preserving birational to a moduli space $M_v(S,H)$ of sheaves on some projective K3 surface $S$;
   \item $V\in \mathcal{C}_d$ for some $d$ satisfying $(**)$.
 \end{enumerate}
\item The following conditions are equivalent:
 \begin{enumerate}
    \item $M_\lambda(V,\tau)$ is stratum preserving birational to a moduli space $M_v(S,\alpha, H)$ of twisted sheaves on some twisted K3 surface $S$;
    \item if $V\in \mathcal{C}_d$ for some $d$ satisfying $(**')$.
 \end{enumerate}
\end{enumerate} 
\end{proposition}

\proof
{\bf (1)} By \cite[Proposition~1.13, Proposition~1.24, Corollary~1.26]{Huybrechts:Gargnano}, if $V\in\mathcal{C}_d$ with $d$ satisfying $(**)$, then there exist a polarised K3 surface $(S,H)$ and an equivalence of categories $\Phi\colon \mathcal{A}_V\to D^b(S)$. This induces a Hodge-isometry $\Phi^K\colon \Ktop(\mathcal{A}_V) \to \Ktop(S)\cong \widetilde{\oH}(S,\mathbb Z)$. If $v:=\Phi^K(\lambda)$,
then clearly we have that the restriction
$$
\Phi^K\colon\lambda^\perp\longrightarrow v^\perp
$$
is a Hodge-isometry as well.

Let us write $v=(v_0,v_2,v_4)\in\widetilde{\oH}(S,\mathbb Z)$.
Without loss of generality, after possibly shifting and taking duals, we may assume that $v_0\geq 0$ and $v_2$ is a non-negative multiple of the ample class $H$. In particular, $v$ is a positive Mukai vector and, by a result of Yoshioka (cf.\ \cite{Yoshioka:ModuliSpaces}) the moduli space $M_v(S,H)$ is non-empty. By Proposition~\ref{prop:v perp} and Proposition~\ref{prop:PR}, we eventually get an isometry
$$\phi\colon\oH^2(M_\lambda(V,\tau),\Z)\cong\lambda^\perp\longrightarrow v^\perp\cong\oH^2(M_v(S,H),\Z)$$
that is an isomorphism of Hodge structures. Now, passing to the symplectic resolutions $\widetilde{M}_\lambda(V,\tau)$ and $\widetilde{M}_v(S,H)$, we get a natural Hodge-isometry
$$\tilde{\phi}\colon \oH^2(\widetilde{M}_\lambda(V,\tau),\Z)\cong\Gamma_\lambda\longrightarrow \Gamma_v\cong\oH^2(\widetilde{M}_v(S,H),\Z)$$
obtained by sending the classes of the exceptional divisors into each other. By the Global Torelli Theorem for manifolds of type OG10 (see \cite[Introduction]{Ono}), we have then that $\widetilde{M}_\lambda(V,\sigma)$ and $\widetilde{M}_v(S,H)$ are birational. Therefore $M_\lambda(V,\tau)$ and $M_v(S,H)$ are also birational and, by construction, the birationality is stratum preserving.

For the converse, assume that $M_\lambda(V,\tau)$ is stratum preserving birational to a Gieseker moduli space $M_v(S,H)$ on a K3 surface $S$; in particular $v = 2w$ with $w^2 = 2$. The birational morphism from $M_\lambda(V,\tau)$ to $M_v(S,H)$ extends to a birational morphism between the desingularisations $\widetilde{M}_\lambda(V,\tau)$ and $\widetilde{M}_v(S,H)$. Since the last two varieties are smooth symplectic varieties, the birational morphism induces a Hodge-isometry on the Beauville--Bogomolov--Fujiki lattices., i.e.\ 
$$\oH^2(\widetilde{M}_\lambda(V,\tau),\Z)\cong\oH^2(\widetilde{M}_v(S,H),\Z)$$ 
and by hypothesis this isometry sends the class of the exceptional divisor to the class of the exceptional divisor. In particular it restricts to a Hodge-isometry 
$$\oH^2(M_\lambda(V,\tau), \mathbb Z) \simeq \oH^2(M_v(S,H), \mathbb Z)$$ 
(cf.\ Proposition~\ref{prop:Gamma} and item $(3)$ of Proposition~\ref{prop:PR}).
By Proposition~\ref{prop:v perp} and item $(2)$ of Proposition~\ref{prop:PR} we get that
$$
K_{\mathrm{top}}(\mathcal{A}_V)\supset \lambda^\perp \simeq \oH^2(M_\lambda(V,\tau), \mathbb Z) \simeq \oH^2(M_v(S,H), \mathbb Z) \simeq v^\perp \subset \widetilde{\oH}(S,\mathbb Z).
$$
Now, since the lattices $\lambda^\perp$ and $v^\perp$ have discriminant group isomorphic to $\Z/2\Z$, this Hodge-isometry must act as the identity on the discriminant group and, by \cite[Corollary~1.5.2]{Nik79}, it extends to a Hodge-isometry $K_{\mathrm{top}}(\mathcal{A}_V) \simeq \widetilde{\oH}(S,\mathbb Z)$ (sending $\lambda$ to $v$).
Finally, $V\in \mathcal{C}_d$ for some $d$ satisfying $(**)$ by \cite[Proposition~1.13, Proposition~1.24]{Huybrechts:Gargnano}.

{\bf (2)} The proof is very similar to the case before, we only remark on the subtle differences. First of all, by \cite[Proposition~1.13, Proposition~1.24]{Huybrechts:Gargnano} and \cite[Proposition~33.1]{BLNPS}, $Y\in\mathcal{C}_d$ with $d$ satisfying $(**')$ if and only if there exist a twisted K3 surface $(S,\alpha)$ and an equivalence of categories $\Phi\colon \mathcal{A}_V\to D^b(S,\alpha)$. Such an equivalence of categories induces a Hodge-isometry $\Phi^K\colon K_{\mathrm{top}}(\mathcal{A}_V) \to \Ktop(S,\alpha)$. Put $v=\Phi^K(\lambda)$ and consider the moduli space $M_v(S,\alpha, H)$. For the non-emptiness of $M_v(S,\alpha)$ one uses \cite{Yoshioka:TwistedSheaves}. The rest of the argument and the reverse implication follow verbatim as in the part $(1)$ of the proof (cf.\ Remark~\ref{rmk:twisted}).
\endproof

Using a lattice-theoretic trick, we can remove the stratum preserving hypothesis, at least in the untwisted case.
\begin{theorem}\label{thm:qualsiasi}
For any smooth cubic fourfold $V$, we consider the singular LPZ variety $M_\lambda(V,\tau)$, where $\lambda=2\lambda_0$, $\lambda_0^2=2$ and $\tau$ is a generic stability condition on the Kuznetsov component $\mathcal{A}_V$. The following conditions are equivalent:
 \begin{enumerate}
   \item $M_\lambda(V,\tau)$ is birational to a moduli space $M_v(S,H)$ of sheaves on some projective K3 surface $S$;
   \item $V\in \mathcal{C}_d$ for some $d$ satisfying $(**)$.
 \end{enumerate}
\end{theorem}
\begin{proof}
One direction follows from Proposition~\ref{thm:main}. So let us suppose that $M_\lambda(V,\tau)$ is birational to a moduli space $M_v(S,H)$ of semistable sheaves on a projective K3 surface $S$. In particular, the desingularised moduli spaces $\widetilde{M}_\lambda(V,\tau)$ and $\widetilde{M}_v(S,H)$ are birational and hence we have an induced Hodge-isometry
$$\Gamma_\lambda=\oH^2(\widetilde{M}_\lambda(V,\tau),\Z)\cong\oH^2(\widetilde{M}_v(S,H),\Z)=\Gamma_v.$$
For a lattice $L$ with a weight $2$ Hodge structure, we denote by $T(L)$ the induced transcendental lattice, defined as the smallest sub-Hodge structure containing $L^{2,0}$. Then there is an induced Hodge-isometry
\[ T(\Gamma_\lambda)\cong T(\Gamma_v). \]
Notice that, by definition, $T(\Gamma_\lambda)=T(\lambda^\perp)$ and $T(\Gamma_v)=T(v^\perp)=T(S)$, where the latter is the transcendental lattice of the K3 surface $S$. Now, the orthogonal complement $T(S)^\perp\subset\widetilde{\oH}(S,\Z)$ contains a unimodular hyperbolic plane $U$, namely the hyperbolic plane generated by $\oH^0(S,\Z)$ and $\oH^4(S,\Z)$. Moreover, by construction we have two primitive embeddings of $T(S)$:
\[ T(S)\hookrightarrow\widetilde{\oH}(S,\Z)\quad\mbox{ and }\quad T(S)\cong T(\lambda^\perp)\hookrightarrow\widetilde{\oH}(\mathcal{A}_V). \]
By \cite[Theorem~1.14.4]{Nik79}, there must exist an isometry \[ g\colon\widetilde{\oH}(\mathcal{A}_V)\longrightarrow \widetilde{\oH}(S,\Z) \]
that preserves the Hodge structures by construction. Then we can conclude as before by \cite[Proposition~1.13, Proposition~1.24]{Huybrechts:Gargnano}.
\end{proof}

\begin{remark}
Both the proposition and the theorem above should be compared with the derived Torelli theorems for (twisted) K3 surfaces. More precisely, two K3 surfaces are derived equivalent if and only if their Mukai lattices are Hodge-isometric if and only if their transcendental lattices are Hodge-isometric (\cite[Theorem~3.3]{Orlov}).

On the other hand, two twisted K3 surfaces are derived equivalent if and only if their Mukai lattices are orientation-preserving Hodge-isometric (\cite[Theorem~B]{Reineke} and \cite[Theorem~0.1]{HuySte}). Here the orientation chosen is the one with respect to the positive 4-space. In this case it seems no longer true that this condition is equivalent to the existence of an isometry between the transcendental lattices (cf.\ \cite[Remark~4.10]{HuySte}).
\end{remark}

\begin{example}\label{example:bir M}
When $V$ is a Pfaffian cubic fourfold, it is known that $V$ has an associated K3 surface $S$. In this case there is natural birational morphism from the manifold $\widetilde{M}_{2\lambda_1+2\lambda_2}(V,\tau)$ of Example~\ref{example:LPZ} and the O'Grady resolution $\widetilde{M}_{(2,0,-2)}(S,H)$, which we now recall. First of all, by \cite[Theorem~1.3]{PertusiCo}, the manifold $\widetilde{M}_{2\lambda_1+2\lambda_2}(V,\tau)$ is birational to a twisted intermediate jacobian fibration (see Section~\ref{section:LSV}). Moreover, by \cite[Example~4.3.6]{OnoPhD}, in this case the twisted intermediate jacobian fibration is isomorphic to the untwisted intermediate jacobian fibration constructed in \cite{LSV}. Finally, in \cite[Section~6]{LSV} the authors construct a birational map between the intermediate jacobian fibration and the manifold $\widetilde{M}_{(2,0,-2)}(S,H)$.

This birational map is explicitely known only at a general point of the smooth locus and its construction is rather difficult. We are not aware of any known description of this morphism at a general point of the singular locus. In particular, the question whether it is stratum preserving is open, and it would be very interesting to have an answer to it.  
\end{example}

\section{When is a LPZ manifold birational to a LSV manifold?}\label{section:LSV}

Let $Y$ be a smooth cubic threefold. The intermediate jacobian of $Y$ is defined as
$$ J_Y=\oH^{2,1}(Y)^*/\oH_3(Y,\Z),$$
where $\oH_3(Y,\Z)$ is included in $\oH^{2,1}(Y)^*$ via integration. $J_Y$ is a principally polarised abelian variety of dimension $5$ and it parametrises dimension 1 cycles on $Y$ that are homologically trivial.

For any $t\in\oH^4(Y,\Z)=\Z$, we denote by $J_Y^t$ the torsor parametrising cycles of homology class $t$. Notice that, up to canonical isomorphism, there exists only one non-trivial torsor, namely $J_Y^1$.

Let now $V$ be a smooth cubic fourfold and let us denote by $\mathcal{U}\subset\P\oH^0(V,\sO_V(1))^*$ the open subset parametrising smooth linear sections. Then there exist two fibrations
\begin{equation}\label{eqn:LSV open}
p\colon\sJ_{\sU}\to\sU\quad\mbox{ and }\quad p^t\colon\sJ_{\sU}^t\to\sU
\end{equation}
whose fibres are of the form $J_Y$ and $J_Y^t$, respectively.
\begin{theorem}[\protect{\cite{LSV,Voisin:Twisted,Sacca}}]\label{thm:Sacca}
There exist smooth and projective compactifications
$$ p\colon\operatorname{IJ}(V)\to(\P^5)^\vee\quad\mbox{ and }\quad  p^t\colon\operatorname{IJ}^t(V)\to(\P^5)^\vee $$
of the fibrations (\ref{eqn:LSV open}). Moreover, both $\operatorname{IJ}(V)$ and $\operatorname{IJ}^t(V)$ are projective irreducible holomorphic symplectic manifolds of type OG10, and both $p$ and $p^t$ are lagrangian fibrations. 
\end{theorem}

Varieties isomorphic to $\operatorname{IJ}(V)$ are called \emph{LSV varieties}; varieties isomorphic to $\operatorname{IJ}^t(V)$ are called \emph{twisted LSV varieties}.

\begin{remark}
Theorem~\ref{thm:Sacca} is an existence result. It is known that when $V$ is very general (in the sense of Hassett), then the compactifications in Theorem~\ref{thm:Sacca} are unique, but this may not longer be true for special cubic fourfolds. In \cite[Section~6]{MO} it is proved that in the twisted case there exists only one compactification that is a lagrangian fibration, under the additional condition that the fibres are irreducible. (In reference to Example~\ref{example:LPZ}, the cases where the stability condition $\bar{\tau}$ is conjecturally non-generic correspond to the cases where the compactifications of the twisted intermediate jacobian have reducible fibres.)

Even though there may be several compactifications, 
by construction all the compactifications are birational to each other. Since we are interested in the birational class of LPZ varieties, we can safely ignore this lack of uniqueness.
\end{remark}




We retain the notation from the previous section, so that $\mathcal{C}$ denotes the moduli space of smooth cubic fourfolds and $\mathcal{C}_d$ denotes the Hassett divisor of special cubic fourfolds with discriminant $d$.
\begin{theorem}\label{thm:LPZ is LSV}
For any smooth cubic fourfold $V$, we consider the desingularised LPZ variety $\widetilde{M}_\lambda(V,\tau)$, where $\lambda=2\lambda_1+2\lambda_2$ and $\tau$ is $\lambda$-generic (cf.\ Example~\ref{example:LPZ}). Then the following statements are equivalent:
\begin{enumerate}
    \item there exists a birational isomorphism between $\widetilde{M}_\lambda(V,\tau)$ and $\operatorname{IJ}(V)$;
    \item $V\in\mathcal{C}_d$ with $d>6$ and $d\equiv 2\pmod{6}$.
\end{enumerate}
\end{theorem}
\proof
First of all, by \cite[Corollary~3.10]{Sacca} (see also \cite[Example~3.3]{MO}), it is known that if $V$ is very general (in the sense of Hassett), then the twisted LSV manifold is not birational to the untwisted one. Since by \cite[Theorem~1.3]{PertusiCo} the manifold $\widetilde{M}_\lambda(V,\tau)$ is always birational to $\operatorname{IJ}^t(V)$, in order to have any birational isomorphism between $\widetilde{M}_\lambda(V,\tau)$ and $\operatorname{IJ}(V)$, the cubic fourfold $V$ must be special, i.e.\ $V\in\mathcal{C}_d$ with $d>6$ and $d\equiv 0,2\pmod{6}$.

Let us assume that $V\in\mathcal{C}_d$ with $d>6$ and $d\equiv 2\pmod{6}$.
In particular this is equivalent to say that $d>6$ is even and $d\equiv 2\pmod{3}$. Let $T\subset V$ be an algebraic two-dimensional cycle (not homologous to a complete intersection) such that the discriminant of the lattice generated by the cohomology class $[T]$ and $h^2$ is $d$. (Notice that such a cycle $T$ exists since $V\in\mathcal{C}_d.)$ Here $h$ is the class of an hyperplane section of $V$ and $h^2$ is the corresponding class in $\operatorname{H}^4(V,\Z)$. If we set $x=[T]^2$ and $y=h^2.[T]$, then $d=3x-y^2$. By hypothesis we must have $y\equiv\pm1\pmod{3}$; in other words, the intersection of $T$ with a smooth linear section of $V$ is a cycle of degree not a multiple of $3$. We can then use $T$ to construct a trivialisation of the torsor $p^t\colon\sJ_{\sU}^t\to\sU$. 
In particular we get an isomorphism between $\sJ_{\sU}^t$ and $\sJ_{\sU}$, which implies that the varieties $\operatorname{IJ}^t(V)$ and $\operatorname{IJ}(V)$ are birational, so that also $\widetilde{M}_\lambda(V,\tau)$ is birational to $\operatorname{IJ}(V)$. 

Let us now prove the other implication. 
Assume that $\widetilde{M}_\lambda(V,\tau)$ is birational to an LSV variety $\operatorname{IJ}(V)$, so that there is an isometry 
\[ \Pic(\widetilde{M}_\lambda(V,\tau))\cong\Pic(\operatorname{IJ}(V)). \]
Since, by \cite[Proposition~4.1]{Ono}, $\Pic(\operatorname{IJ}(V))$ contains a unimodular hyperbolic plane, also $\Pic(\widetilde{M}_\lambda(V,\tau))$ contains a unimodular hyperbolic plane.

On the other hand, as we noticed in Example~\ref{example:Pic}, we have
\[ \Pic(\widetilde{M}_\lambda(V,\tau))=\langle U(3),\operatorname{H}^{2,2}(V,\Z)_{\operatorname{prim}}\rangle, \]
where $U(3)$ is the primitive sublattice generated by the exceptional divisor $\widetilde{\Sigma}$ and the movable isotropic class of the lagrangian fibration induced on $\widetilde{M}_\lambda(V,\tau)$ by \cite[Theorem~1.3]{PertusiCo}.

Since the lattice $U(3)$ is primitive in $\Pic(\widetilde{M}_\lambda(V,\tau))$ and $\operatorname{H}^{2,2}(V,\Z)_{\operatorname{prim}}$ is negative definite, the only way for $\Pic(\widetilde{M}_\lambda(V,\tau))$ to contain a hyperbolic plane is that $\operatorname{H}^{2,2}(V,\Z)_{\operatorname{prim}}$ contains at least a class of divisibility $3$.
By \cite[Proposition~3.2.2]{Hassett} this can happen only if $V\in\mathcal{C}_d$ with $d\equiv 2\pmod{6}$.
\endproof

\begin{example}
If $V$ is a Pfaffian cubic fourfold, so that $V\in\mathcal{C}_{14}$, as already remarked in \cite[Example~4.3.6]{OnoPhD} $\operatorname{IJ}(V)$ and $\operatorname{IJ}^t(V)$ are isomorphic. In particular the manifold $\widetilde{M}_{2\lambda_1+2\lambda_2}(V,\tau)$ is birational to $\operatorname{IJ}(V)$ and moreover
\[
\Pic(\widetilde{M}_{2\lambda_1+2\lambda_2}(V,\tau))=U\oplus\langle D\rangle, 
\]
where $D$ is a class of square $-42$ and divisibility $3$.
\end{example}

\begin{example}\label{example:bah}
    More generally let us take a cubic fourfold $V\in\mathcal{C}_d$ with $d>6$ and $d\equiv 2\pmod{6}$, and let us assume that $V$ is general in $\mathcal{C}_d$. In particular $\oH^{2,2}(V,\Z)_{\operatorname{prim}}=\Z D$ where $D^2=-3d$ and $\operatorname{div}(D)=3$ (see \cite[Proposition~3.2.2]{Hassett}). Since $d$ is always even, let us write $d=2k$, so that $D^2=-6k$.

    In this case we have that
    \[ \Pic(\widetilde{M}_\lambda(V,\tau))=\langle \bar{e},\bar{f},D\rangle, \]
    where $\bar{e}$ and $\bar{f}$ are the standard basis of $U(3)$. The rank $3$ lattice generated by these three classes has the following Gram matrix
    \[
    \begin{pmatrix}
     0 & 3 & 0\\
     3 & 0 & 0\\
     0 & 0 & -6k,
     \end{pmatrix} \]
     so that it is easy to see that 
     \[ A:=\frac{\bar{e}+k\bar{f}+D}{3}\in\Pic(\widetilde{M}_\lambda(V,\tau)). \]
     Notice that $A^2=0$, $A.\bar{f}=1$ and $A.D=-2k$, so that we eventually get that the rank $2$ lattice generated by $A$ and $\bar{f}$ is the unimodular hyperbolic plane and moreover the class $Z=D+2k\bar{f}$ is ortogonal to it (and it has divisibility $3$). Eventually we get that
     \[ \Pic(\widetilde{M}_\lambda(V,\tau))=\langle A,\bar{f},Z\rangle=
     \begin{pmatrix}
     0 & 1 & 0\\
     1 & 0 & 0\\
     0 & 0 & -6k.
     \end{pmatrix} \]

     Vice versa, if $\operatorname{div}(D)=1$, then there cannot exist any isometric embedding of $U$ in $\Pic(\widetilde{M}_\lambda(V,\tau))$. In fact if $x\in\oH^2(\widetilde{M}_\lambda(V,\tau),\Z)$ is a class orthogonal to $U(3)$ and such that $D.x=1$, then we will have $A.x=1/3$, which is absurd.
\end{example}




\end{document}